\documentclass[12pt,reqno]{amsart}

\usepackage{amsmath,amsthm,amssymb,comment,fullpage}

\usepackage{times}

\usepackage[T1]{fontenc}

\usepackage{mathrsfs}

\usepackage{latexsym}

\usepackage[dvips]{graphics}

\usepackage{epsfig}

\usepackage{amsmath,amsfonts,amsthm,amssymb,amscd}

\input amssym.def

\input amssym.tex

\usepackage{color}

\usepackage{hyperref}

\usepackage{url}

\usepackage{breakurl}

\newcommand{\bburl}[1]{\textcolor{blue}{\url{#1}}}

\numberwithin{equation}{section}

\newtheorem{thm}{Theorem}[section]

\newtheorem{cor}[thm]{Corollary}

\newtheorem{lem}[thm]{Lemma}

\newtheorem{prop}[thm]{Proposition}

\newtheorem{defn}[thm]{Definition}

\newtheorem{cla}[thm]{Claim}

\theoremstyle{plain}

\newtheorem{corollary}[thm]{Corollary}

\newtheorem{example}[thm]{Example}

\newtheorem{lemma}[thm]{Lemma}

\newtheorem{proposition}[thm]{Proposition}

\newtheorem{theorem}[thm]{Theorem}

\newtheorem*{theorem*}{Theorem}

\newtheorem{conjecture}[thm]{Conjecture}

\newcommand\be{\begin{equation}}

\newcommand\ee{\end{equation}}

\newcommand\bea{\begin{eqnarray}}

\newcommand\eea{\end{eqnarray}}

\newcommand\bi{\begin{itemize}}

\newcommand\ei{\end{itemize}}

\newcommand\ben{\begin{enumerate}}

\newcommand\een{\end{enumerate}}

\newcommand\bc{\begin{center}}

\newcommand\ec{\end{center}}

\newcommand\ba{\begin{array}}

\newcommand\ea{\end{array}}

\newcommand{\tbf}[1]{\textbf{#1}}

\newcommand{\R}{\ensuremath{\mathbb{R}}}

\newcommand{\Z}{\ensuremath{\mathbb{Z}}}

\newcommand\frakfamily{\usefont{U}{yfrak}{m}{n}}

\DeclareTextFontCommand{\textfrak}{\frakfamily}

\newcommand{\hr}[1]{\href{#1}{\url{#1}}}

\title{The Emergence of 4-Cycles in Polynomial Maps over the Extended Integers}

\author{Andrew Best}

\email{\textcolor{blue}{\href{mailto:best.221@osu.edu)}{best.221@osu.edu}}}

\address{Department of Mathematics and Statistics, Williams College, Williamstown, MA 01267}

\curraddr{Department of Mathematics, The Ohio State University, Columbus, OH 43210}

\author{Patrick Dynes}

\email{\textcolor{blue}{\href{mailto:pdynes@clemson.edu}{pdynes@clemson.edu}}}

\address{Department of Mathematical Sciences, Clemson University, Clemson, SC 29634}

\author{Steven J. Miller}

\email{\textcolor{blue}{\href{mailto:sjm1@williams.edu}{sjm1@williams.edu}}}

\address{Department of Mathematics and Statistics, Williams College, Williamstown, MA 01267}

\author{Jasmine Powell}

\email{\textcolor{blue}{\href{mailto:jasminepowell2015@u.northwestern.edu}{jasminepowell2015@u.northwestern.edu}}}

\address{Department of Mathematics, Northwestern University, Evanston, IL 60208}

\author{Benjamin Weiss}

\email{\textcolor{blue}{\href{mailto:benjamin.weiss@maine.edu}{benjamin.weiss@maine.edu}}}

\address{Department of Mathematics and Statistics, The University of Maine, Orono, ME 04469}

\subjclass[2010]{11D99 (primary), 11A41 (secondary), 11B37 (secondary).} 

\keywords{Polynomial maps, arithmetic dynamics, cycles, extended integers}

\thanks{The first, second and fourth named authors were supported by NSF Grant DMS1347804 and Williams College, and the third named author was partially supported by NSF Grant DMS1265673. We thank Conor Hetland for assistance with some of the coding, and Karl Winsor for helpful conversations. We also thank Michael Zieve for his comments on an earlier version, and bringing the paper by Narkiewicz to our attention. }

\date{\today}

\begin{document}

\maketitle








\begin{abstract}

Let $f(x) \in \Z[x]$; for each integer $\alpha$ it is interesting to consider the number of iterates $n_{\alpha}$, if possible, needed to satisfy $f^{n_{\alpha}}(\alpha) =  \alpha$. The sets $\{\alpha, f(\alpha), \ldots, f^{n_{\alpha} - 1}(\alpha), \alpha\}$ generated by the iterates of $f$ are called cycles. For $\Z[x]$ it is known that cycles of length 1 and 2 occur, and no others. While much is known for extensions to number fields, we concentrate on extending $\Z$ by adjoining reciprocals of primes. Let $\Z[1/p_1, \ldots, 1/p_n]$ denote $\Z$ extended by adding in the reciprocals of the $n$ primes $p_1, \ldots, p_n$ and all their products and powers with each other and the elements of $\Z$.


Interestingly, cycles of length 4, called 4-cycles, emerge for polynomials in $\Z\left[1/p_1, \ldots, 1/p_n\right][x]$ under the appropriate conditions. The problem of finding criteria under which 4-cycles emerge is equivalent to determining how often a sum of four terms is zero, where the terms are $\pm 1$ times a product of elements from the list of $n$ primes. We investigate conditions on sets of primes under which 4-cycles emerge. We characterize when 4-cycles emerge if the set has one or two primes, and (assuming a generalization of the ABC conjecture) find conditions on sets of primes guaranteed not to cause 4-cycles to emerge.

\end{abstract}

\maketitle

\tableofcontents








\section{Introduction}

\subsection{Background and Motivation}


Let $R$ be a ring and $f(x) \in R[x]$ a polynomial over $R$. For any fixed $\alpha \in R$ we define the cycle starting at $\alpha$ to be the sequence $(\alpha, f(\alpha)$, $f(f(\alpha)) := f^2(\alpha)$, $\dots)$. When this cycle consists of only finitely many distinct elements of $R$, then $\alpha$ is said to be pre-periodic. We study the cases where $\alpha$ is periodic; that is, where $f^{n}(\alpha) = \alpha$ for some integer $n$.

\begin{defn}

Given a ring $R$ and a polynomial $f$ in $R[x]$, an \textbf{$n$-cycle} (or a \textbf{cycle of length $n$}) is a sequence of $n$ distinct elements of the ring, $(x_1, \dots, x_n)$, such that \begin{equation} f(x_1) \ =  \ x_2,\ \ \  f(x_2) \ = \ x_3,\ \ \  \dots, \ \ \ f(x_n) \ = \ x_1.\end{equation} \end{defn}


It is well known that when $R = \Z$ the only possible cycle lengths are 1 and 2, both of which occur. For one proof, see \cite[Lemma 28]{Zieve}. In more generality, the possible cycle lengths for a polynomial in a number field has been related to the unit group of the ring of integers, see \cite{Lenstra}. In his thesis Zieve \cite{Zieve} showed that if $R = \Z_{(2)}$, the localization\footnote{That is, $\Z_{(2)}$ consists of all fractions where the numerator is an integer and the denominator is not divisible by 2.} of $\Z$ at the ideal $(2)$, then the only possible cycle lengths are 1, 2, and 4. It is thus natural to consider rings properly contained between $\Z$ and $\Z_{(2)}$. In particular, we are interested in the rings $\Z \left[1/p_1, \dots, 1/p_n \right]$ which are formed by adjoining the reciprocals of $n$ odd primes $\{p_1, \ldots, p_n\}$ along with all their products and powers with each other and the elements of $\Z$. We call $\{p_1,\ldots, p_n\}$ the \emph{\textbf{inversion set}} associated to $\Z \left[1/p_1, \ldots, 1/p_n \right]$. Because $\Z \subset \Z \left[1/p_1, \ldots, 1/p_n \right]$, these intermediary rings of course have cycles of length 1 and 2.

While it is not known which rings $\Z \left[1/p_1, \dots, 1/p_n \right]$ have polynomials that exhibit 4-cycles, there is an elegant connection between the existence of 4-cycles in a ring $\Z \left[1/p_1\right.$, $\dots$, $\left.1/p_n \right]$ and the solvability of special equations involving products of the primes in its inversion set.

\begin{lem} If there is a polynomial in $\Z \left[1/p_1, \ldots, 1/p_n \right]$ that exhibits a 4-cycle, then we can write

\be u_1 \ + \ u_2 \ + \ u_3 \ + \ u_4\ = \ 0, \ee

where $u_i = \pm p_1^{a_{i1}} \cdots p_n^{a_{in}}$ and each $a_{ij}$ is a nonnegative integer. In order to discard pathological examples like $1 - 1 + 1 - 1 = 0$ or $p - p + 1 - 1 = 0$, we also insist that the $u_i$'s have no proper subsum equal to 0.

\end{lem}

This fact is a consequence of Corollary 20 in \cite{Zieve}, and we now provide a paraphrasing of an explanation from his thesis.




To show the necessity of the existence of such a linear relation


suppose that $f \in R[x]$ has the 4-cycle $(x_1, x_2, x_3, x_4)$. As the polynomial $x-y$ divides $f(x) - f(y)$ in $R[x,y]$, we find

\begin{align}
x_{i} \ - \ x_{i-1} \ | \ f(x_{i}) \ - \ f(x_{i-1}) \ \ = \ \ x_{i+1} - x_{i}, \hspace{15 mm} 1 \leq i \leq 4.
\end{align}



\noindent From this, we obtain the chain of divisors

\begin{align}
x_2 \ - \ x_1 \ | \ x_3 \ - \ x_2 \ | \ x_4 \ - \ x_3 \ | \ x_1 \ - \ x_4 \ | \ x_2 \ - \ x_1.
\end{align}

This shows that the pairwise ratios of $u_1 = x_2  -  x_1, \ \ u_2 = x_3  -   x_2, \ \ u_3 = x_4 -   x_3,$ and $u_4 = x_1  -  x_4$ are units in $R$. Therefore, if an orbit exists, we are guaranteed a sum of units equaling zero.

Note that this condition is not sufficient for there to be a 4-cycle; see Lemma \ref{lem:cor27}.








\begin{defn} \label{defn:AdmitCycle}

We say that a set of primes $\{p_1, \dots, p_n\}$ \tbf{admits a 4-cycle} if we can write $\epsilon_1 u_1 \ + \ \epsilon_2 u_2 \ + \ \epsilon_3 u_3 \ + \ \epsilon_4 u_4 = 0$ with $\epsilon_i \in \{-1, 1\}$ and $u_i = p_1^{a_{i1}} \cdots p_n^{a_{in}}$ where each $a_{ij}$ is a nonnegative integer; we further require that the $u_i$'s have no zero proper subsum. If a set of primes does not admit a 4-cycle, we say it \tbf{avoids a 4-cycle}. Moreover, we say that this set \tbf{linearly admits (or avoids) a 4-cycle} if each $a_{ij} \in \{0, 1\}$, and in general, we say that this set \tbf{admits (or avoids) a 4-cycle with $n$-powers} if each $a_{ij} \in \{0,1,\ldots,n\}$.

\end{defn}



We have justified in a natural way the requirements of Definition \ref{defn:AdmitCycle} in this section. With only a minor abuse of notation, we apply the same terminology for sets of primes to the ring $R$.









\subsection{Summary of Main Results}

We first attempt to classify inversion sets of low cardinality by whether they admit 4-cycles. Theorems \ref{prop:singleton} and \ref{thm:doubletonclass} also appear in Narkiewicz \cite{Nar}[Theorem 1 \& Theorem 2] with similar proofs.

\begin{theorem}\label{prop:singleton}

$\Z[1/p]$ admits a 4-cycle if and only if $p = 2$ or $3$.

\end{theorem}

\begin{theorem}\label{thm:doubletonclass} Fix a positive integer $n$. An inversion set with two elements admits a 4-cycle if any of the following hold:

\begin{enumerate}

\item it is of the form $\{p,p+2\}$, with $p$ and $p+2$ both prime,

\item it is of the form $\{p,p^n-2\}$, with $p$ and $p^n-2$ both prime,

\item it is of the form $\{p,2p+1\}$, with $p$ and $2p+1$ both prime.

\end{enumerate}

\end{theorem}

We prove related results for infinite sets, such as Corollary \ref{cor:upperdensity4cycle} (which states that any inversion set with positive upper density not only admits a 4-cycle, but does so linearly).

We then turn to the much harder problem of constructing inversion sets that are proven to avoid 4-cycles. Our main result assumes a generalized ABC conjecture. If we do not assume this conjecture we can prove that certain sets avoid 4-cycles with $n$ powers (i.e., no prime occurs to a power greater than $n$); see \S\ref{subsec:avoiding4cycles} for detailed constructions.

\begin{theorem}\label{t:main}

If Conjecture \ref{c:BrowBrz} is true, then there exist infinitely many pairs of distinct primes $p_1$ and $p_2$ such that $\Z\left[\frac{1}{p_1}, \frac{1}{p_2}\right]$ does not have a 4-cycle.

\end{theorem}

After proving some useful auxiliary results, we prove the above theorems in \S\ref{sec:main}, and give conditions on the two primes in Theorem \ref{t:main} that, under Conjecture \ref{c:BrowBrz} holding, ensure there is no 4-cycle. We conclude with a discussion of some future research problems in \S\ref{sec:futurework} and some examples in the appendices.








\section{Proofs of Main Results}\label{sec:main}

We begin with a result of Zieve that will be useful throughout the paper.

\begin{lemma}\label{lem:cor27}(Corollary 27, \cite{Zieve}) Let $R$ be an integral domain. There exists a polynomial in $R[x]$ having a 4-cycle in $R$ if and only if there exist units $u$ and $v$ for which $u+v$ and $u+1$ are associates, and for which $1+u+v$ is a unit.

\end{lemma}

This leads us to a partial reformulation of the problem of characterizing sets of primes that admit a 4-cycle.

\begin{proposition}

If the set of primes $\{p_1, \dots, p_n\}$ does not admit a 4-cycle, then the ring $\Z\left[1/p_1\right.$, $\ldots$, $\left.1/p_n \right]$ has no polynomial with a 4-cycle.

\end{proposition}

\begin{proof}

By means of contraposition, assume that $R = \Z\left[1/p_1, \ldots, 1/p_n \right]$ has a polynomial with a 4-cycle. Then by Lemma \ref{lem:cor27} we know that there exist units $u, v, w \in \R$ such that $1 + u + v = w$. Units in $R$ are of the form $p_1^{a_1} \cdots p_n^{a_n}$ with $a_1, \dots, a_n \in \Z$. We multiply through to eliminate negative exponents on the primes, which yields an equation of the form

\begin{equation}\epsilon_1t_1 \ + \ \epsilon_2 t_2 \ + \ \epsilon_3 t_3 \ + \ \epsilon_4 t_4 \\ = \ \ 0, \end{equation}

where $\epsilon_i \ \in \ \{-1, 1\}$ and $t_i \ = \ p_1^{a_1} \cdots p_n^{a_n}$ with each $a_i$ a positive integer. Therefore, we see that $\{p_1, \dots, p_n\}$ admits a 4-cycle.

\end{proof}

\subsection{Singleton Inversion Sets}

We turn to the proof of Theorem \ref{prop:singleton}, which states a singleton inversion set $\Z[1/p]$ admits a 4-cycle if and only if $p = 2$ or $3$.




\begin{proof}[Proof of Theorem \ref{prop:singleton}]

First, consider the case when $p = 2$. Let $R = \Z \left[ 1/2 \right], u = 2, v = 1$. Then, by Lemma \ref{lem:cor27}, $\Z\left[ 1/2 \right]$ admits a 4-cycle.

Next, consider the case when $p = 3$. Letting $R = \Z \left[ 1/3 \right]$ and $u = v = 1$, by Lemma \ref{lem:cor27} we now have that $\Z\left[ 1/3 \right] $ admits a 4-cycle.  

Otherwise, let $p > 3$ be a prime. We know that $\Z \left[1/p \right]$ admits a 4-cycle if and only if there exist values of $a_i$ such that

\begin{equation}\pm p^{a_1} \pm p^{a_2} \pm p^{a_3} \pm p^{a_4}\ = \ 0\end{equation}

and this equation admits no zero proper subsum.

By multiplying by the appropriate power of $p$, namely $p^{- \min a_i}$, we rewrite the equation as

\begin{equation}1 \pm p^{b_1} \pm p^{b_2} \pm p^{b_3}\ = \ 0, \end{equation}

where $b_1, b_2, b_3 \geq 0$ and at least one sign is negative. Note that, disregarding solutions to this equation that admit a zero proper subsum, looking at this equation $\bmod \, p$ we have either

\begin{equation}1 \equiv 0, \; 2 \equiv 0, \text{ or } 3 \equiv 0 \bmod p,\end{equation}

depending on the number of $b_i = 0$. However, since $p > 3$, this is a contradiction. Therefore $\Z \left[1/p \right]$ admits a 4-cycle if and only if $p = 2$ or $3$.

\end{proof}

\begin{example}

The polynomial $f(x) = -\frac{2}{3}x^3 + 4x^2 - \frac{19}{3} x + 5$ has the 4-cycle $(1,2,3,4)$ in $\Z \left[ 1/3 \right]$.

\end{example}

The proposition above answers for almost all rings $\Z\left[1/p \right]$ (except $p = 2$) the question of which cycle lengths are allowed. The case of $p = 2$ is handled completely by Narkiewicz \cite{Nar}[Lemmas 5 \& 6]. In addition, Narkiewicz provides every example of a polynomial in $\Z[\frac{1}{2}]$ with a $4$-cycle.

\subsection{Other Inversion Sets Admitting 4-Cycles} \label{sec:doubletonsadmit4cycles}


As soon as we consider slightly larger inversion sets, say of cardinality 2, the picture turns murky. Using our reformulation of the problem in the introduction, it is often a matter of algebra to find inversion sets with special structure that admit 4-cycles. We give three examples in Theorem \ref{thm:doubletonclass} (restated below for convenience), with full explanation for the first, and suggest others.\\ \

\noindent \textbf{Theorem \ref{thm:doubletonclass}} (Partial Classification of Doubleton Inversion Sets) \emph{Fix a positive integer $n$. An inversion set with two elements admits a 4-cycle if any of the following hold:
\begin{enumerate}
\item it is of the form $\{p,p+2\}$, with $p$ and $p+2$ both prime,
\item it is of the form $\{p,p^n-2\}$, with $p$ and $p^n-2$ both prime,
\item it is of the form $\{p,2p+1\}$, with $p$ and $2p+1$ both prime.
\end{enumerate}
}



\begin{proof}[Proof of (1)] Using our reformulation, $\{p,p+2\}$ admits a 4-cycle if we can write

\begin{equation} u_1 \ + \ u_2 + \ u_3 \ + \ u_4 \ \ = \ \ 0 \end{equation}

with $u_i = \pm p^{a_{i1}}(p+2)^{a_{i2}}$, each $a_{ij}$ a nonnegative integer, and the set of $u_i$'s has no zero proper subsum.

Write

\begin{align}
u_1 \ & = \ p+2 \nonumber \\
u_2 \ & = \ -1 \nonumber \\
u_3 \ & = \ -1 \nonumber \\
u_4 \ & = \ -p.
\end{align}

The result follows.

\end{proof}

Similar proofs for the other cases are given in Appendix \ref{app:doubletonclass}.

To actually construct a polynomial $f(x)$ in $\Z[1/p,1/(p+2)][x]$ that admits a 4-cycle, simply choose a 4-cycle $(x_1,x_2,x_3,x_4)$ with appropriate step sizes $u_i$. Using Lagrange interpolation with $f(x_1) \ = \ x_2$, and so on, one can construct $f$. 

\begin{example}[Example of (1)] \label{4cyclearbtriv} Consider the polynomial

\begin{equation} f(x) \ \ = \ \ -\frac{2}{35}x^3 \ - \ \frac{4}{35}x^2 \ + \ \frac{221}{35}x \ + \ \frac{101}{7}\ \in\ \Z\left[\frac 15, \frac 17 \right][x]. \end{equation}

It is easy to verify that $f$ has the 4-cycle $(-10,-3,-4,-9)$. This example also shows that the step sizes $u_i$ may be reordered.

\end{example}

Other interesting inversion sets of size 2 that admit 4-cycles exist. There are also larger inversion sets that admit 4-cycles. In fact, it is trivial to find inversion sets of arbitrary size that admit 4-cycles: Simply add in primes to an inversion set of size 2 that already admits 4-cycles. For example, to get an inversion set of size 3 that admits 4-cycles, you might consider the polynomial $f$ in Example \ref{4cyclearbtriv} viewed as an element of the ring $\Z [1/5,1/7,1/37][x]$. Thus, the bulk of our paper is dedicated to investigating inversion sets that \textbf{avoid} 4-cycles, as this problem is more interesting. 




\subsection{Separations and Avoiding 4-Cycles}\label{subsec:avoiding4cycles}


The following theorem from Green is the starting point of our investigations of inversion sets avoiding 4-cycles.




\begin{theorem}[Theorem 1.4, \cite{Green}] \label{thm:Green}

Write $\mathcal{P}$ for the set of primes. Every subset of $\mathcal{P}$ of positive upper density contains a 3-term arithmetic progression.

\end{theorem}




Green's result immediately implies the following useful characterization.

\begin{prop}

If a set of primes does not linearly admit a 4-cycle, then it has density 0 in the primes.

\end{prop}

\begin{proof}

Suppose we have a set of primes with a 3-term arithmetic progression; that is, we have distinct primes $p_1, \ p_2, \ p_3$ such that $p_2 = p_1 + a$ and $p_3 = p_2 + 2a$ where $a \in \Z$. Then

\begin{align}\label{eq:arithprog}
p_3 - p_2 - p_2 + p_1  \ = \  p_1 + 2 a - (p_1 + a) - (p_1 + a) + p_1  \ = \  0,
\end{align} and this set of primes linearly admits a 4-cycle. Thus we have that if a set of primes does not linearly admit a 4-cycle, then it contains no arithmetic progressions. Then, by Theorem \ref{thm:Green}, we have that any set of primes that does not linearly admit a 4-cycle has density 0 in the primes.

\end{proof}

Note that this does not guarantee that $u+v$ and $u+1$ will be associates, to satisfy the conditions of Lemma \ref{lem:cor27}.


%


%


\begin{cor}\label{cor:upperdensity4cycle}

Every subset of $\mathcal{P}$ of positive upper density must linearly admit a 4-cycle.

\end{cor}


%


%

%

We now construct sets of primes that linearly avoid 4-cycles. We consider equations of the form

\begin{align} \label{eq:FourTerms}
\epsilon_1t_1 + \epsilon_2 t_2 + \epsilon_3 t_3 + \epsilon_4 t_4 & \ = \  0, \ \ \ \epsilon_i \in \{-1, 1\},\ \ \ t_i  \ = \  p_1^{a_{i1}} \cdots p_n^{a_{in}}, \ \ \ a_{ij} \in \{0, 1\}
\end{align}

and discount trivial solutions, that is, instances where the set $\{\epsilon_i t_i\}$ contains a proper subsum equal to 0, as discussed in Definition \ref{defn:AdmitCycle}.


Without loss of generality, let the term $t_1$  be maximal in the set of terms $\{t_1, t_2, t_3, t_4\}$. If $t_1 > t_2 + t_3 + t_4$ then it is easy to see that $\epsilon_1t_1 + \epsilon_2 t_2 + \epsilon_3 t_3 + \epsilon_4 t_4 \neq 0$ for all choices of $\epsilon_i$.

\begin{lemma}[Separation Lemma]\label{sep_lemma}  Suppose we have an ordering of positive terms $\{x_1,x_2,\ldots,x_n\}$ such that $x_{i-1} < x_i$ holds for each $2 \leq i \leq n$. Also suppose that given an $x_i$, for any three distinct terms $x_{j_1}, x_{j_2}, x_{j_3} < x_i$ we have $x_i > x_{j_1} + x_{j_2} + x_{j_3}$. Then there is no set of $t_i$'s that non-trivially satisfy

\[ \epsilon_1t_1 + \epsilon_2 t_2 + \epsilon_3 t_3 + \epsilon_4 t_4  \ = \  0,
\]

where $t_i \in \{x_i\}$ and $\epsilon_i = \{-1,1\}$.

\end{lemma}

\begin{proof} It suffices to show the result holds for the ordering of terms $\{x_1, x_2, x_3, x_4\}$ such that $x_{i-1} < x_{i}$ for each $2 \leq i \leq 4$ and $x_4 > x_1 + x_2 + x_3$. In this case, let $t_i = x_i$ for $1 \leq i \leq 4$. Then $t_1 + t_2 + t_3 - t_4 < 0$. It is now easy to see that $\epsilon_i t_i + \epsilon_i t_i + \epsilon_i t_i + \epsilon_i t_i \neq 0$ for all choices of $\epsilon_i$.

\end{proof}

With Lemma $\ref{sep_lemma}$, we characterize a set of primes which linearly avoids a 4-cycle in the next result. In particular, we note that the conditions of the lemma are satisfied if $x_i > 3 x_{i-1}$.




\begin{prop}

Fix a positive integer $k$, and let $p_1 > 3, \;  p_j > 3\prod_{i=1}^{j-1}p_i$ for $2 \le j \le k$ be prime. Then the inversion set $\{p_1,p_2,\ldots,p_k\}$ linearly avoids a 4-cycle. 

\end{prop}

\begin{proof}

We establish that for any two terms $s,t$ drawn from the set of products $\{ p_1^{\alpha_1}\cdots p_{n}^{\alpha_{n}} : \alpha_i \in \{0,1\} \}$, if $s < t$ then $3s < t$, giving enough separation to linearly avoid a 4-cycle.

We proceed by induction on the number of primes in the inversion set. If we only have one prime $p > 3$, then $\{p\}$ linearly (and in fact generally by Proposition \ref{prop:singleton}) avoids a 4-cycle. Now suppose that a set of $k$ primes $\{p_1, \dots, p_k\}$ with the separation properties above linearly avoids a 4-cycle. Consider adding in another prime $p_{k+1}$, satisfying $p_{k+1} > 3p_1 \cdots p_k$. Let

\[ S \ := \  \{p_1^{\alpha_1}\cdots p_{k}^{\alpha_{k}} : \alpha_i \in \{0,1\} \}
\]

denote the set of products for which the induction hypothesis applies, and let

\[ S^* \ := \  \{p_1^{\alpha_1}\cdots p_{k+1}^{\alpha_{k+1}} : \alpha_i \in \{0,1\} \}
\]

denote the set of products extended with the possibility of a $p_{k+1}$ factor.

\begin{cla}

For all $s, t \in S^*$, if $s < t$ then $3s < t$.

\end{cla}

To see this, choose $s = p_1^{\alpha_1}\cdots p_{k+1}^{\alpha_{k+1}}$ and $t = p_1^{\beta_1}\cdots p_{k+1}^{\beta_{k+1}}$ in $S^*$ such that $s < t$.

First, if $\alpha_{k+1} = 0$ and $\beta_{n+1} = 0$, then $s, t \in S$, so that $3s < t$ by hypothesis.

Next, if instead we have $\alpha_{k+1} = 1$ and $\beta_{k+1} = 1$, then $s/p_{k+1},t/p_{k+1} \in S$ with $s/p_{k+1} < t/p_{k+1}$, so that by hypothesis we have $3s/p_{k+1} <  t/p_{k+1}$, so that $3s < t$.

Further, suppose $\alpha_{k+1} = 1$ and $\beta_{k+1} = 0$. Then $s \geq p_{k+1} > 3p_1\cdots p_k > t$, which contradicts $s < t$.

Finally, suppose that $\alpha_{k+1} = 0$ and $\beta_{k+1} = 1$. Then we have that $3s \leq 3p_1\cdots p_k < p_{k+1} \leq t$, so that $3s < t$. The claim follows.

Therefore, following the discussion above, for any $k$, provided the separation conditions are met, the inversion set $\{p_1,p_2,\ldots,p_k\}$ linearly avoids a 4-cycle.

\end{proof}

We have discussed the case in which the powers of primes are restricted to first powers (linear avoidance). We now consider the case where we allow powers of primes up to some integer $n$. The ideas behind the proof are very similar.




\begin{prop}

Fix positive integers $n$ and $k$, and let $p_1 > 3, \; p_j > 3\prod_{i=1}^{j-1}p_i^n$ for $2\le j \le k$ be prime. Then the inversion set $\{p_1, p_2, \ldots, p_k\}$ avoids a 4-cycle with $n$-powers. 

\end{prop}

\begin{proof}

We establish that for any two terms $s,t$ drawn from the set of products $\{ p_1^{\alpha_1}\cdots p_{n}^{\alpha_{n}} : \alpha_i \in \{0,1,\ldots,n\} \}$, if $s < t$ then $3s < t$, giving enough separation to avoid a 4-cycle with $n$-powers.

We proceed by induction on the number of primes in the inversion set. If we only have one prime $p > 3$, then $\{p\}$ avoids a 4-cycle by Proposition \ref{prop:singleton}). Now suppose that a set of $k$ primes $\{p_1, \dots, p_k\}$ with the separation properties above avoids a 4-cycle with $n$-powers. Consider adding in another prime $p_{k+1}$, satisfying $p_{k+1} > 3p_1^n \cdots p_k^n$. Let

\[ S \ := \  \{p_1^{\alpha_1}\cdots p_{k}^{\alpha_{k}} : \alpha_i \in \{0,1,\ldots,n\} \}
\]

denote the set of products for which the induction hypothesis applies, and let

\[ S^* \ := \  \{p_1^{\alpha_1}\cdots p_{k+1}^{\alpha_{k+1}} : \alpha_i \in \{0,1,\ldots,n\} \}
\]

denote the set of products extended with the possibility of a $p_{k+1}$ factor.

\begin{cla}

For all $s, t \in S^*$, if $s < t$ then $3s < t$.

\end{cla}

To see this, choose $s = p_1^{\alpha_1}\cdots p_{k+1}^{\alpha_{k+1}}$ and $t = p_1^{\beta_1}\cdots p_{k+1}^{\beta_{k_1}}$ in $S^*$ such that $s < t$.

First, if $\alpha_{k+1} = 0$ and $\beta_{k+1} = 0$, then $s,t \in S$, so that $3s < t$ by hypothesis.

Next, if instead we have $\alpha_{k+1} = \beta_{n+1} \neq 0$, then $s/p_{k+1}^{\alpha_{k+1}}, t/p_{k+1}^{\alpha_{k+1}} \in S$ with $s/p_{k+1}^{\alpha_{k+1}} < t/p_{k+1}^{\alpha_{k+1}}$, so that by hypothesis we have $3s/p_{k+1}^{\alpha_{k+1}} < t/p_{k+1}^{\alpha_{k+1}}$, so that $3s < t$.

Further, suppose $\alpha_{k+1} > \beta_{k+1}$. Then we have

\[ s\ \geq\ p_{k+1}^{\alpha_{k+1}}\ >\ 3p_1^n\cdots p_k^n p_{k+1}^{\alpha_{k+1} - 1}\ \geq\ 3p_1^n\cdots p_k^n p_{k+1}^{\beta_{k+1}}\ >\ p_1^n\cdots p_k^n p_{k+1}^{\beta_{k+1}}\ \geq\ t,
\]

which contradicts $s < t$.

Finally, suppose that $\alpha_{k+1} < \beta_{k+1}$. Then we have

\[ 3s\ \leq\ 3p_1^n\cdots p_k^n p_{k+1}^{\alpha_{k+1}}\ <\ p_{k+1}^{\alpha_{k+1}+1}  \ \leq \  p_{k+1}^{\beta_{k+1}}  \ \leq \  t,
\]

so that $3s < t$.

Therefore, following the discussion above, for any $k$ and $n$, provided the separation conditions are met, the inversion set $\{p_1,p_2,\ldots,p_k\}$ avoids a 4-cycle with $n$-powers.

\end{proof}

Up until now, we have restricted the powers we allow on primes when constructing sets that avoid 4-cycles to some degree. We remove this restriction of possible powers after first proving the following helpful lemma.





\begin{lemma}\label{lem:primeordering}



Let $m > 7$ be an integer.

Fix primes $p_1$ and  $p_2$ such that for all non-negative integers $k$ and all integers $\ell$ with $0\le \ell < m$ it holds that \[p_1^{\ell + k}\ <\ p_1^k p_2^\ell\ <\ \frac{1}{3} p_1^{\ell + \frac{\ell}{m} + k}.\]  Then for any non-negative integers $\ell, k, s,$ and $t$ satisfying \begin{enumerate}\item $0 \le t <m$;\item $0 \le \ell < m$;\item  $t(1 + \frac{1}{m}) + s > \ell(1 + \frac{1}{m}) + k$,\end{enumerate} it holds that $p_1^s p_2^t > p_1^kp_2^{\ell}$.

\end{lemma}

\begin{proof}



%

%

%

The lemma is clear whenever $t + s > \ell\left(1 + \frac{1}{m}\right) + k$, as $p_1^s p_2^t > p_1^{t + s} > p_1^{\ell\left(1 + \frac{1}{m}\right) + k} > p_1^kp_2^\ell$.

Otherwise, we must have that $t(1 + \frac{1}{m}) + s > \ell(1+\frac{1}{m}) + k \ge t+s$. Since both $\frac{t}{m}$ and $\frac{\ell}{m}$ are less than $1$, it must be that $t + s = k + \ell.$ In this case, since $t(1 + \frac{1}{m}) + s > \ell(1 + \frac{1}{m}) + k$, we have that $t/m > \ell/m$ and so $t > \ell$. Therefore, $k = s + t- \ell$. Thus,

\begin{equation}p_1^sp_2^t  \ = \  p_1^sp_2^{t-\ell}p_2^\ell  > p_1^s p_1^{t-\ell}p_2^\ell  \ = \  p_1^k p_2^\ell.\end{equation}



\end{proof}









%



%




For our main theorems, we need to assume the following generalization of the ABC conjecture.

\begin{conjecture}\label{c:BrowBrz}[Browkin-Brzezinski]

Given an integer $n > 2$ and an $\epsilon> 0$, there exists a constant $C_{n, \epsilon}$, such that for all integers $a_1, \dots, a_n $ with $a_1+\cdots+ a_n=0$ (and no proper subset having a zero sum), and $\gcd( a_1, \dots, a_n)=1$ we have

\begin{equation}\label{eq:generalabc}
\max(|a_1|,\dots,|a_n|)   \ \leq \  C_{n,\epsilon} ({\rm rad}(a_1 \times \cdots \times a_n))^{2n-5+\epsilon},
\end{equation}

where ${\rm rad}(n)$ is the product of the distinct prime factors of $n$.

\end{conjecture}





\begin{proof}[Proof of Theorem \ref{t:main}]

We may assume without loss of generality that $p_2 > p_1$, and that $a_1 + a_2 + a_3 + a_4 = 0$ is a sum of powers of $p_1$ and $p_2$ which vanish, and no sum with fewer terms is zero. Our proof deals with two cases: the first has at least one summand divisible by a large power of $p_2$, and the second assumes that all four summands are only divisible by a small number of multiples of $p_2$.

In applying the conjecture, we are concerned about the case when $n = 4$, and we fix $\epsilon = 1$ for simplicity. Note that each $a_i$ is a product of powers of $p_1$ and $p_2$ and so ${\rm rad}(a_1 \times \cdots \times a_n) = p_1 p_2$. We then get

\begin{equation}\label{eq:ourgeneralabc}
\max(|a_1|,|a_2|,|a_3|,|a_4|)\ \leq\ C_{4, 1} (p_1 p_2)^{4}.
\end{equation}

Choose $m \in \Z$ such that $m > \max\{8, \log_3(C_{4,1})\}$ and choose primes $p_1$ and $p_2$ such that the following conditions are satisfied:

\begin{enumerate}

\item $p_1 \geq 18^m$;

\item $p_2 \in \left(3p_1, \frac{1}{3}p_1^{1+\frac{1}{m}}\right)$.

\end{enumerate}

We first show that we can always find a prime satisfying Condition 2. From Condition 1 we have that $p_1 \geq 18^m$, which we can obviously satisfy for any choice of $m$. We then get

\bea\label{eq:prime1conditions}
p_1\ \geq\ 18^m\  \ \ \
%
%
\Rightarrow \ \ \ \  \frac{1}{3} p_1^{1 + \frac{1}{m}}\ \geq\ 2 \cdot 3p_1.
\eea


By Bertrand's postulate (see for example \cite{Da}) there is always a prime in $(x, 2x)$ for all $x > 1$. Thus we see that we can find a prime $p_2 \in (3p_1, 2\cdot 3p_1) \subseteq \left(3p_1, \frac{1}{3}p_1^{1+\frac{1}{m}}\right)$.

We first deal with the case where at least one summand is divisible by a large power of $p_2$; in particular, we define large to be greater than $m$. 

%




%

To show that any term divisible by $p_2^k$ where $k \geq m$ cannot be a part of a 4-term sum, we only need to show that $p_2^{m} > C_{4,1}(p_1 p_2)^{4}$.

We required that $m > \max\{8, \log_3(C_{4,1})\}$. This gives us that $C_{4,1} < 3^m$, 

because $m > \log_3(C_{4,1})$;

it also gives that $\frac{8m + 4}{m^2} < 1$, because $m$ is at least $9$. 

Combining these inequalities, we see that


\begin{equation} C_{4,1}^\frac{1}{m} p_1^{\frac{8m + 4}{m^2}}\ <\ 3p_1.\end{equation}

Finally, since $3p_1 < p_2$, we get that

\begin{equation} p_2\ >\ C_{4,1}^\frac{1}{m} p_1^{\frac{8m + 4}{m^2}}\end{equation}

and so

\bea
p_2^m & \ \ge \ & C_{4,1} p_1^{\left(8 + \frac{4}{m}\right)} \nonumber \\
&>& C_{4,1} p_1^4 p_1^{4\left({1 + \frac{1}{m}}\right)} \nonumber \\
&>& C_{4,1} (p_1 p_2)^4,
\eea


the last substitution coming from the fact that $p_2 < p_1^{1+\frac{1}{m}}.$

Thus if a 4-term sum exists, it cannot have any term divisible by $p_2^m$.

Next, we invoke Lemma \ref{lem:primeordering} to show that when all the terms are divisible only by powers of $p_2$ that are less than $m$, we have enough separation between possible products of powers of primes to make the sum impossible.

To see this, we need to show that letting $S =  \{ p_1^k p_2^\ell | k \in \Z_{\geq 0}, 0 \leq \ell < m\}$ that for all $a, b \in S$ if $a < b$ then $3a < b$.

Our lemma gives us an ordering on the elements -- now that we have this ordering, we only need to verify the two following cases.

\begin{enumerate}

\item $3p_1^kp_2^\ell < p_1^{k-1}p_2^{\ell+1}$

\item $3p_2^\ell < p_1^{\ell + 1}$.

\end{enumerate}








Case 1 corresponds to the case in which $a = p_1^k p_2^\ell$ and $b =a \frac{p_2}{p_1}$. By our original conditions, we have that $\frac{p_2}{p_1} > 3$ and therefore $3a < b$.

In Case 2 we have $a = p_2^\ell$ and $b = p_1^{\ell+1}$, and by our initial conditions, since $\ell \leq m$, we see that

\begin{equation} p_2^\ell\ <\ \frac{1}{3}p_1^{\ell + \frac{\ell}{m}} \end{equation}

so

\begin{equation} 3p_2^\ell\ <\ p_1^{\ell + 1}.\end{equation}

Therefore, for all $a, b \in S$, if $a < b$ then $3a < b$.


With all of these conditions in place, we then see that $\Z\left[1/p_1, 1/p_2\right]$ does not have a 4-cycle.

\end{proof}

\subsection{Numerics}

We end by examining patterns that occur when counting the number of cycles for a given prime list. Based on our observations and the formulation of our results, we conjecture that the number of cycles that occur when considering a specific list of primes correlates with the spacing between the primes. Intuitively, if the primes are spaced far apart, the likelihood of them ``interacting'' in a way that gives a cycle -- that is, finding some combination of four products of the primes that sums to zero -- is small. We examine this conjecture through computation and find the pattern to hold.

Figure \ref{fig:mingap} gives a plot of the number of cycles based on the minimum gap between primes in the inversion set. The points are based on lists of five of the first 50 primes. For example, the inversion set $\{37, 73, 83, 127, 157\}$ admits two cycles. The minimum gap associated to this list is 10. In the plots, the size of the point at any given position represents the number of lists associated with that gap and number of cycles.

\begin{center}
\begin{figure}[h]
\includegraphics[scale = 0.5]{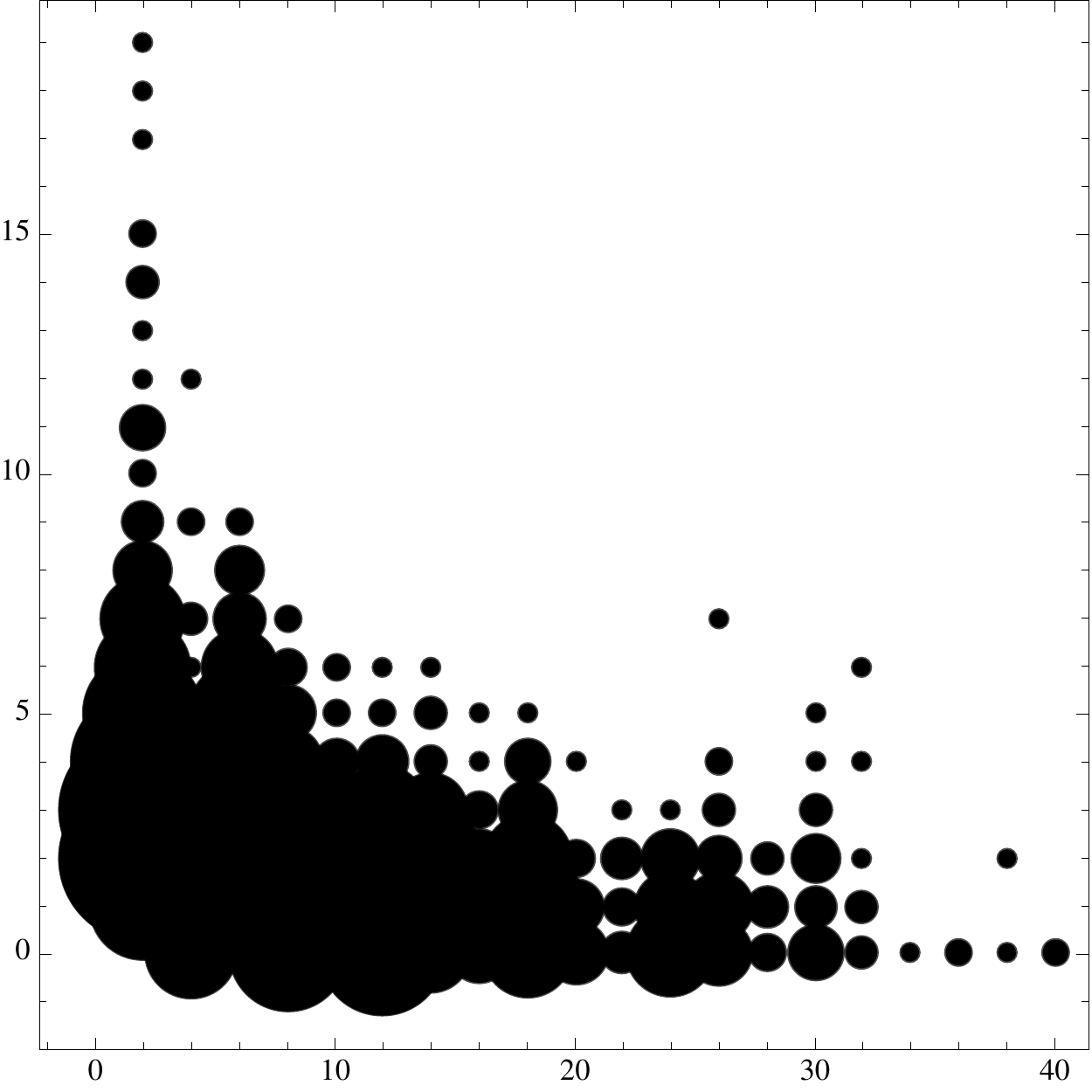}
\caption{\label{fig:mingap} Minimum gap between inversion set vs. number of cycles, with size of points reflecting density of the data.}
\end{figure}
\end{center}




\section{Future Work}\label{sec:futurework}

In terms of the main result, there are at least two directions in which to proceed. First, we would like to extend Theorem \ref{t:main} to sets of $n$ primes. The main difficulty with extending our method of proof is constructing a set of primes so that both methods in the proof still apply. In particular, a generalization of Lemma \ref{lem:primeordering} would be needed. Second, we would like to eliminate the dependence of Theorem \ref{t:main} on the generalized ABC conjecture given as Conjecture \ref{c:BrowBrz}.

In Section \ref{sec:doubletonsadmit4cycles} we considered particular shapes of doubleton inversion sets that admit 4-cycles. Related questions of the following flavor suggest themselves: Given an inversion set consisting of a particular odd prime $p$, what is the minimal number of primes we need to add to the set to ensure that it admits a 4-cycle? Per the results in that subsection, this answer might usually be one (if $p$ happens to be a twin prime, for example), but occasionally it might be two. Or, if we have an inversion set with two primes that is known to avoid 4-cycles, how ``easy'' is it to introduce 4-cycles by adding primes to it? (Maybe such sets happen to be difficult to disrupt in this way, and require essentially grafting an entire inversion set that is known to work, such as a pair of twin primes, or maybe not.) These and other questions would be interesting to consider further.






\appendix

\section{Cycle Lengths in $\Z[1/2]$}\label{app:p=2case}




A few results from \cite{Zieve} will be helpful. The \tbf{Lenstra constant} of a ring $R$ was defined in \cite{Lenstra} to be

\begin{align}
L(R) =& \sup \{k : \text{there exist $x_1, \ldots, x_k \in \R$ such that $x_i - x_j \in R^{*}$} \\
& \text{for all $i, j$ for which $1 \leq i < j \leq k$ } \nonumber \}
\end{align}


\begin{example} $L(\Z) = 2$. To show that $L(\Z) \geq 2$, just consider the set $\{0,1\}$. To see that the Lenstra constant cannot exceed 2, without loss of generality we can shift all our elements so the first is 0. As the units are $\pm 1$, without loss of generality $x_2 = 1$, and there is no choice for $x_3$ such that $x_3 - 0$ and $x_3 - 1$ are both units.

\end{example}

\begin{lem}[Lemma 22, \cite{Zieve}] \label{lem:Lem22Zieve}

If a polynomial over $R$ has a $p$-cycle in $R$, where $p$ is prime, then $p \leq L(R)$.

\end{lem}



%



%

%




%




%

%

%

First we prove a helpful lemma.




%

\begin{lemma}\label{lem:Z2nobigpcycles}

$\Z \left[ 1/2 \right]$ admits no cycles of prime length $p > 3$.

\end{lemma}

\begin{proof}

We first compute the Lenstra constant $L\left(\Z\left[ 1/2 \right]\right)$. Note that $f(x) = -(3/2) x^{2} + (11/2)x - 2$ has the 3-cycle $(1,2,3)$, so by Lemma \ref{lem:Lem22Zieve}, $L\left(\Z \left[ 1/2 \right]\right) \geq 3$.

Then, assume to the contrary that $L\left(\Z \left[ 1/2 \right]\right) \geq 4$; that is, assume there exist $x_1, \dots , x_4 \in \Z \left[ 1/2 \right]$ such that $x_i - x_j \in \Z \left[ 1/2 \right]^*$ for all $i,j$ for which $1 \leq i <j \leq 4$. Then, for some $k_1, k_2, k_3 \in \Z$ we have the following:

\be x_1 - x_2\ =\ 2^{k_1} \ee

\be x_2 - x_3 \ =\ 2^{k_2} \ee

\be x_3 - x_4 \ =\ 2^{k_3} \ee

\begin{equation} \label{eqn:ktrick12}
x_1 - x_3 = 2^{k_1}+2^{k_2} \ =\ 2^{k_1}(2^{k_2-k_1}+1)
\end{equation}

\begin{equation} \label{eqn:ktrick23}
x_2 - x_4 = 2^{k_2}+2^{k_3} \ =\ 2^{k_2}(2^{k_3-k_2}+1). \end{equation}

Equation \eqref{eqn:ktrick12} implies that $k_2 - k_1 = 0$, for otherwise, $x_1 - x_3$ would not be a unit. Similarly, Equation \eqref{eqn:ktrick23} implies that $k_3 - k_2 = 0$, so that $k_1 = k_2 = k_3$. Then we have that

\begin{equation}
x_1 - x_4 = 2^{k_1}+2^{k_2}+2^{k_3} = 3 \cdot 2^{k_1},
\end{equation}

so that $x_1 - x_4$ is not a unit, which is the desired contradiction. Thus, $L\left(\Z\left[ 1/2 \right]\right) < 4$, so that $L\left(\Z\left[ 1/2 \right]\right) = 3$.

Finally, by Corollary 24 in \cite{Zieve}, the only cycles of prime length that are admitted are of length 2 or 3. The result follows.

\end{proof}


We can obtain a slightly stronger result by considering the 3-smooth numbers.

\begin{defn} Let $B$ be a fixed integer. An integer $n$ is said to be \emph{\textbf{$B$-smooth}} if none of its prime factors are larger than $B$. That is, if $p$ is prime and $p \ | \ n$, then $p \leq B$.

\end{defn}




\begin{corollary}

If $\Z \left[ 1/2 \right]$ admits a cycle of length $k$, then $k$ is 3-smooth.

\end{corollary}

\begin{proof}

Assume to the contrary that $f(x) \in \Z\left[ 1/2 \right][x]$ has a cycle $(a_1,a_2,\ldots,a_k)$ of length $k$, with $k$ not 3-smooth. Then there exists a prime $p > 3$ such that $p \ | \ k$. But then $f^\frac{k}{p}(x) \in \Z\left[ 1/2 \right][x]$ has a cycle of length $p$, namely $(a_{k/p},a_{2k/p},\ldots,a_k)$, which contradicts Lemma \ref{lem:Z2nobigpcycles}.

\end{proof}




\section{Proofs and Examples of Theorem \ref{thm:doubletonclass} } \label{app:doubletonclass}

\begin{proof}[Proof of Theorem \ref{thm:doubletonclass} (2)] Using our reformulation, we write

\begin{align}
u_1 \ & = \ p^n-2 \nonumber \\
u_2 \ & = \ 1 \nonumber \\
u_3 \ & = \ 1 \nonumber \\
u_4 \ & = \ -p^n.
\end{align}

It is clear that $u_1 \ + \ u_2 \ + \ u_3 \ + \ u_4 = 0$ and there are no zero proper subsums of $u_i$'s, so that $\{p,p^n-2\}$ admits a 4-cycle.

\end{proof}

We give an example of this case when $n=2$ and $p=5$.

\begin{example}[Example of (2)] Consider the polynomial

\begin{equation}
g(x) \  \ = \  \ -\frac{2}{575}x^3 \ + \ \frac{112}{115}x^2 \ + \ \frac{3127}{575}x \ - \ \frac{16019}{115} \ \in \Z\left[\frac 15, \frac{1}{23}\right][x].
\end{equation}

It is easy to verify that $g$ has the 4-cycle $(-14,-15,10,9)$. This example also shows that the step sizes $u_i$ can appear with all polarities reversed, too.

\end{example}

\begin{proof}[Proof of Theorem \ref{thm:doubletonclass} (3)] Using our reformulation, we write

\begin{align}
u_1 \ & = \ 2p+1 \nonumber \\
u_2 \ & = \ -p \nonumber \\
u_3 \ & = \ -p \nonumber \\
u_4 \ & = \ -1.
\end{align}

It is clear that $u_1 \ + \ u_2 \ + \ u_3 \ + \ u_4 = 0$ and the set of $u_i$'s has no zero proper subsum, so that $\{p,2p+1\}$ admits a 4-cycle.

\end{proof}

\begin{example}[Example of (3)] Consider the polynomial

\begin{equation}
h(x) \  \ = \  \ -\frac{2}{11}x^3 \ - \ \frac{146}{55}x^2 \ - \ \frac{39}{5}x \ + \ 7/11 \ \in \Z\left[\frac 15, \frac {1}{11}\right][x].
\end{equation}

It is easy to verify that $h$ has the 4-cycle $(-10,-5,-4,1)$.

\end{example}








\end{document}